\documentclass[fleqn, a4paper, oneside]{imsart}

\usepackage{amsthm,amsmath,amssymb}
\usepackage{graphicx}

\usepackage[authoryear]{natbib}

\bibpunct{(}{)}{;}{a}{,}{,}
\RequirePackage[OT1]{fontenc}

\startlocaldefs

\newtheorem{theorem}{Theorem}

\newtheorem{lemma}{Lemma}

\theoremstyle{remark}
\newtheorem{remark}{Remark}

\theoremstyle{definition}
\newtheorem{definition}{Definition}

\endlocaldefs

\begin{document}

\begin{frontmatter}


\title{Detection of objects in noisy images and site percolation on square lattices}
\runtitle{Detection in noisy images and percolation.}


\begin{aug}
\author{\snm{Mikhail} \fnms{Langovoy}\corref{}\thanksref{t2}
\ead[label=e1]{langovoy@eurandom.tue.nl}}

\affiliation{
        EURANDOM,\\
         The Netherlands.}

\address{Mikhail Langovoy, Technische Universiteit Eindhoven, \\
EURANDOM, P.O. Box 513,
\\
5600 MB, Eindhoven, The Netherlands\\
\printead{e1}\\
Phone: (+31) (40) 247 - 8113\\
Fax: (+31) (40) 247 - 8190\\}

\and

\author{\snm{Olaf} \fnms{Wittich} \ead[label=e2]{o.wittich@tue.nl}}

\affiliation{
        Technische Universiteit Eindhoven and EURANDOM,\\
         The Netherlands.}

\address{Olaf Wittich, Technische Universiteit Eindhoven and \\
EURANDOM, P.O. Box 513,
\\
5600 MB, Eindhoven, The Netherlands\\
\printead{e2}\\
Phone: (+31) (40) 247 - 2499}

\thankstext{t2}{Corresponding author.}

\runauthor{M. Langovoy and O. Wittich}
\end{aug}

\begin{abstract}
We propose a novel probabilistic method for detection of objects in noisy images. The method uses results from percolation and random graph theories. We present an algorithm that allows to detect objects of unknown shapes in the presence of random noise. Our procedure substantially differs from wavelets-based algorithms. The algorithm has linear complexity and exponential accuracy and is appropriate for real-time systems. We prove results on consistency and algorithmic complexity of our procedure.\\
\end{abstract}


\begin{keyword}
\kwd{Image analysis} \kwd{signal detection} \kwd{image reconstruction} \kwd{percolation} \kwd{noisy image}
\end{keyword}

\end{frontmatter}

\section{Introduction}\label{Section1}

In this paper, we propose a new efficient technique for quick detection of objects in noisy images. Our approach uses mathematical percolation theory.

Detection of objects in noisy images is the most basic problem of image analysis. Indeed, when one looks at a noisy image, the first question to ask is whether there is any object at all. This is also a primary question of interest in such diverse fields as, for example, cancer detection (\cite{Cancer_Detection_1}), automated urban analysis (\cite{Road_Detection_IEEE}), detection of cracks in buried pipes (\cite{Sinha200658}), and other possible applications in astronomy, electron microscopy and neurology. Moreover, if there is just a random noise in the picture, it doesn't make sense to run computationally intensive procedures for image reconstruction for this particular picture. Surprisingly, the vast majority of image analysis methods, both in statistics and in engineering, skip this stage and start immediately with image reconstruction.


The crucial difference of our method is that we do not impose any shape or smoothness assumptions on the \emph{boundary} of the object. This permits the detection of nonsmooth, irregular or disconnected objects in noisy images, under very mild assumptions on the object's interior. This is especially suitable, for example, if one has to detect a highly irregular non-convex object in a noisy image. Although our detection procedure works for regular images as well, it is precisely the class of irregular images with unknown shape where our method can be very advantageous.

Many modern methods of object detection, especially the ones that are used by practitioners in medical image analysis require to perform at least a preliminary reconstruction of the image in order for an object to be detected. This usually makes such methods difficult for a rigorous analysis of performance and for error control. Our approach is free from this drawback. Even though some papers work with a similar setup (see \cite{Arias-Castro_etal}), both our approach and our results differ substantially from this and other studies of the subject. We also do not use any wavelet-based techniques in the present paper.

We view the object detection problem as a nonparametric hypothesis testing problem within the class of discrete statistical inverse problems.

In this paper, we propose an algorithmic solution for this nonparametric hypothesis testing problem. We prove that our algorithm has linear complexity in terms of the number of pixels on the screen, and this procedure is not only asymptotically consistent, but on top of that has accuracy that grows exponentially with the "number of pixels" in the object of detection. The algorithm has a built-in data-driven stopping rule, so there is no need in human assistance to stop the algorithm at an appropriate step.

In this paper, we assume that the original image is black-and-white and that the noisy image is grayscale. While our focusing on grayscale images could have been a serious limitation in case of image reconstruction, it essentially does not affect the scope of applications in the case of object detection. Indeed, in the vast majority of problems, an object that has to be detected either has (on the picture under analysis) a color that differs from the background colours (for example, in roads detection), or has the same colour but of a very different intensity, or at least an object has a relatively thick boundary that differs in colour from the background. Moreover, in practical applications one often has some prior information about colours of both the object of interest and of the background. When this is the case, the method of the present paper is applicable after simple rescaling of colour values.

The paper is organized as follows. In Section \ref{Section_Model_Assumptions} we describe some general assumptions that one makes in statistical processing of digital images. The statistical model itself is described in details in Section \ref{Section_Model}. Suitable thresholding for noisy images is crucial in our method and is developed in Section \ref{Section_Thresholding}. Our algorithm for object detection is presented in Section \ref{Section2}. Theorem \ref{Theorem1} is the main result about consistency and computational complexity of the algorithm. Section \ref{Section_Proofs} is devoted to the proof of the main theorem.

\section{Basic framework}\label{Section_Model_Assumptions}

In this section we discuss some natural and very basic assumptions that we impose on our model.

Suppose we have an analogous two-dimensional image. For numerical or graphical processing of images on computers, the image always has to be discretized. This is achieved via a \emph{pixelization} procedure.

A typical example of pixelization procedure is as follows. Consider an $N \times N$ grid on the square containing the image, and color black those and only those pixels for whose the pixel's \emph{interior} has common points with the original image. The result is called a pixelized picture.

After certain procedure of pixelization, each pixel gets a real number attached to it. This assumption means that our only data available are given as an array of $N^2$ real numbers ${\{Y_{ij}\}}_{i,j=1}^N$. In order to perform statistical image analysis, we will use only these $N^2$ numbers plus our model assumptions. This leads us to the following basis assumption.\smallskip


$\langle  \textbf{A1} \rangle\quad$ Assume that we have a square screen of $N \times N$ pixels, where we \par
\noindent\quad\quad\quad\quad\quad observe pixelized images. We assume that we are getting our \par
\noindent\quad\quad\quad\quad\quad information about the image from this screen alone. \par\smallskip




In the present paper we are interested in detection of objects that have a \emph{known} colour. And this colour is different from the colour of the background. Mathematically, this can be roughly formulated as follows.\smallskip

$\langle  \textbf{A2} \rangle\quad$ The true (non-noisy) images are black-and-white. \par\smallskip

Indeed, we are free to assume that all the pixels that belong to the meaningful object within the digitalized image have the value 1 attached to them. We can call this value a \emph{black colour}. Additionally, assume that the value 0 is attached to those and only those pixels that do not belong to the object in the \emph{true} image. If the number 0 is attached to the pixel, we call this pixel \emph{white}.
\smallskip

In this paper we always assume that we observe a noisy image. The noise itself can be caused, for example, by channel defects or other transmission errors, distortions, etc. Medical scans provide a classic example of noisy images: the scan is always made indirectly, through the body. In astronomy, one also observes noisy images: there are optical and technological errors, atmosphere conditions, etc. The observed values on pixels could be different from 0 and 1. This means that we will actually always have a greyscale image in the beginning of our analysis. \smallskip

$\langle  \textbf{A3} \rangle\quad$ On each pixel we have random noise that has the distribution function\par

\noindent\quad\quad\quad\quad\quad $F$, where $F$ has mean 0 and known variance $\sigma > 0$; the noise at \par
\noindent\quad\quad\quad\quad\quad each pixel is completely independent from noises on other pixels.
\smallskip

\noindent An important special case is when the noise is normally distributed, i.e. $F = N(0, \sigma)$, where $N(0, \sigma)$ stands for the normal distribution with mean 0 and known variance $\sigma > 0$. However, in general, the noise doesn't need to be smooth, symmetric or even continuous.
\smallskip

\begin{remark}\label{Remark1}
We limit ourselves to two-dimensional images only. However, our method makes it possible to analyze also $n-$dimensional images, for each $n \geq 1$.

\end{remark}
\smallskip

\begin{remark}\label{Remark2}
It doesn't really matter if the screen is square or rectangular. We consider a square screen only to simplify our notation.
\end{remark}
\smallskip

\begin{remark}\label{Remark3} It is behind the scope of the present paper to discuss various ways to pixelize real-world analogous images. One possible method of pixelization is often used in literature (see \cite{Arias-Castro_etal} and related references).

The way of pixelization plays an important role mostly for those problems in image analysis where one needs to give asymptotic estimates for the boundary of the object. For example, consider pixelizing a plane curve. Then one can consider also colouring black those pixels that have at least 4 common points with the curve (not necessarily 4 interior points), etc.  However, for simpler problems like image detection (or crude estimation of object's shape or interior) it is often not important how many pixels you colour at the boundary of your image.
\end{remark}
\smallskip

\begin{remark}\label{Remark4}
Our method works not only for normal or i.i.d. noise. We have chosen this type of noise in order to achieve relatively simple and explicit results. The method itself allows the treatment of more complicated situations, such as singular noise, discrete noise, pixel colour flipping, different noise in different screen areas, and dependent noise.
\end{remark}
\smallskip

%

\section{Statistical Model}\label{Section_Model}

Now we are able to formulate the model more formally. We have an $N \times N$ array of observations, i.e. we observe $N^2$ real numbers ${\{Y_{ij}\}}_{i,j=1}^N$. Denote the true value on the pixel $(i, j)$, $1 \leq i, j \leq N$, by $Im_{ij}$, and the corresponding noise by $\sigma \varepsilon_{ij}$. Therefore, by our model assumptions,

\begin{equation}\label{1}
Y_{ij} = Im_{ij} + \sigma \varepsilon_{ij}\,,
\end{equation}

\noindent where $1 \leq i, j \leq N$, $\sigma > 0$ and, in accordance with assumption $\langle A2 \rangle$,

\begin{equation}\label{3}
Im_{ij} = \left\{
           \begin{array}{ll}
             1, & \hbox{if $(i,j)$ belongs to the object;} \\
             0, & \hbox{if $(i,j)$ does not belong to the object.}
           \end{array}
         \right.
\end{equation}

\noindent To stress the dependence on the noise level $\sigma$, we write assumption $\langle A3 \rangle$ in the following way:

\begin{equation}\label{2}
\varepsilon_{ij} \sim F, \quad \mathbb{E}\, \varepsilon_{ij} = 0,  \quad Var\, \varepsilon_{ij} = 1\,.
\end{equation}

\noindent The noise here doesn't need to be smooth, symmetric or even continuous. Moreover, all the results below are easily transferred to the even more general case when the noise has arbitrary but known distribution function $F_{gen}$; it is not necessary that the noise has mean $0$ and finite variance. The only adjustment to be made is to replace in all the statements quantities of the form $F \bigr(\,\cdot\,/ \sigma\bigr)$ by the quantities $F_{gen}(\,\cdot\,)$. The Algorithm 1 below and the main Theorem \ref{Theorem1} are valid without any changes for a general noise distribution $F_{gen}$ satisfying (\ref{8}) and (\ref{9}).

Now we can proceed to preliminary quantitative estimates. If a pixel $(i, j)$ is white in the original image, let us denote the corresponding probability distribution of $Y_{ij}$ by $P_0$. For a black pixel $(i, j)$ we denote the corresponding distribution of $Y_{ij}$ by $P_1$. We are free to omit dependency of $P_0$ and $P_1$ on $i$ and $j$ in our notation, since all the noises are independent and identically distributed.


\begin{lemma}\label{Lemma1}
Suppose pixel $(i, j)$ has white colour in the original image. Then for all $y \in \mathbb{R}$:

\begin{equation}\label{4}
P_0 (\,Y_{ij} \geq y\, ) = 1 - F \biggr(\frac{y}{\sigma}\biggr)\,,
\end{equation}

\noindent where $F$ is the distribution function of the standardized noise.
\end{lemma}


\begin{proof} (Lemma \ref{Lemma1}): By (\ref{2}),

\[
P_0 (\,Y_{ij} \geq y\, ) = 1 - P(\sigma \varepsilon_{ij} < y ) = 1 - F \biggr(\frac{y}{\sigma}\biggr)\,.
\]

\end{proof}


\begin{lemma}\label{Lemma2}
Suppose pixel $(i, j)$ has black colour in the original image. Then for all $y \in \mathbb{R}$:

\begin{equation}\label{5}
P_1 (\,Y_{ij} \leq y\, ) = F \biggr(\frac{y-1}{\sigma}\biggr)\,.
\end{equation}

\end{lemma}

\begin{proof} (Lemma \ref{Lemma2}): By (\ref{2}) again, we have

\[
P_1 (\,Y_{ij} \leq y\, ) = P(1 + \sigma \varepsilon_{ij} \leq y ) = P(\sigma \varepsilon_{ij} \leq y-1 ) = F \biggr(\frac{y-1}{\sigma}\biggr)\,.
\]

\end{proof}

\section{Thresholding and Graphs of Images}\label{Section_Thresholding}

Now we are ready to describe one of the main ingredients of our method: the \emph{thresholding}. The idea of the thresholding is as follows: in the noisy grayscale image ${\{Y_{ij}\}}_{i,j=1}^N$, we pick some pixels that look as if their real colour was black. Then we colour all those pixels black, irrespectively of the exact value of grey that was observed on them. We take into account the intensity of grey observed at those pixels only once, in the beginning of our procedures. The idea is to think that some pixel "seems to have a black colour" when it is not very likely to obtain the observed grey value when adding a "reasonable" noise to a white pixel.

We colour white all the pixels that weren't coloured black at the previous step. At the end of this procedure, we would have a transformed vector of 0's and 1's, call it $\{\overline{Y}_{i,j}\}_{i,j=1}^{N}$. We will be able to analyse this transformed picture by using certain results from the mathematical theory of percolation. This is the main goal of the present paper. But first we have to give more details about the thresholding procedure.

Let us fix, for each $N$, a real number $\alpha_0(N) > 0$, $\alpha_0(N) \leq 1$, such that there exists $\theta (N) \in \mathbb{R}$ satisfying the following condition:

\begin{equation}\label{6}
P_0 (\,Y_{ij} \geq \theta (N)\, ) \,\leq\, \alpha_0(N)\,.
\end{equation}


\begin{lemma}\label{Lemma3}
Assume that (\ref{6}) is satisfied for some $\theta (N) \in \mathbb{R}$. Then for the smallest possible $\theta (N)$ satisfying (\ref{6}) it holds that

\begin{equation}\label{7}
F \biggr(\frac{\,\theta (N)}{\sigma}\biggr)\,=\, 1 - \alpha_0(N)\,.
\end{equation}
\end{lemma}

\begin{proof} (Lemma \ref{Lemma3}): Obvious by Lemma \ref{Lemma1}.
\end{proof}

In this paper we will always pick $\alpha_0(N)\,\equiv\,\alpha_0$ for all $N \in \mathbb{N}$, for some constant $\alpha_0 > 0$. But we will need to have varying $\alpha_0(\,\cdot)$ for our future research.

We are prepared to describe our thresholding principle formally. Let $p_{c}^{site}$ be the \emph{critical probability for site percolation} on $\mathbb{Z}^2$ (see \cite{Grimmett} for definitions).

As a first step, we transform the observed noisy image $\{Y_{i,j}\}_{i,j=1}^{N}$ in the following way: for all $1 \leq i, j \leq N$,\par\smallskip

1. $\quad$ If $Y_{ij} \geq \theta (N)$, set $\overline{Y}_{ij} := 1$ (i.e., in the transformed picture the corresponding pixel is coloured black).\smallskip

2. $\quad$ If $Y_{ij} < \theta (N)$, set $\overline{Y}_{ij} := 0$ (i.e., in the transformed picture the corresponding pixel is coloured white).\smallskip

\begin{definition}\label{Definition1}
The above transformation is called \emph{thresholding at the level} $\theta (N)$. The resulting vector $\{\overline{Y}_{i,j}\}_{i,j=1}^{N}$ of $N^2$ values (0's and 1's) is called a \emph{thresholded picture}.
\end{definition}
\smallskip

Suppose for a moment that we are given the original black and white image without noise. One can think of pixels from the original picture as of vertices of a planar graph. Furthermore, let us colour these $N^2$ vertices with the same colours as the corresponding pixels of the original image. We obtain a graph $G$ with $N^2$ black or white vertices and (so far) no edges.



We add edges to $G$ in the following way. If any two \emph{black} vertices are neighbours (i.e. the corresponding pixels have a common side), we connect these two vertices with a \emph{black} edge. If any two \emph{white} vertices are neighbours, we connect them with a \emph{white} edge. We will not add any edges between non-neighbouring points, and we will not connect vertices of different colours to each other.

Finally, we see that it is possible to view our black and white pixelized picture as a collection of black and white "clusters" on the very specific planar graph (a square $N \times N$ subset of the $\mathbb{Z}^2$ lattice).

\begin{definition}\label{Definition2}
We call graph $G$ the \emph{graph of the (pure) picture}.
\end{definition}

This is a very special planar graph, so there are many efficient algorithms to work with black and white components of the graph. Potentially, they could be used to efficiently process the picture. However, the above representation of the image as a graph is lost when one considers noisy images: because of the presence of random noise, we get many gray pixels. So, the above construction doesn't make sense anymore. We overcome this obstacle with the help of the above thresholding procedure.

We make $\theta (N)-$thresholding of the noisy image $\{Y_{i,j}\}_{i,j=1}^{N}$ as in Definition \ref{Definition1}, but with a very special value of $\theta (N)$. Our goal is to choose $\theta (N)$ (and corresponding $\alpha_0(N)$, see (\ref{6})) such that:

\begin{eqnarray}
  1 - F \biggr(\frac{\theta(N)}{\sigma}\biggr) &<& p_{c}^{site}\,,\label{8} \\
  p_{c}^{site} &<& 1 - F \biggr(\frac{\theta(N) -1 }{\sigma}\biggr)\,,\label{9}
\end{eqnarray}

\noindent where $p_{c}^{site}$ is the critical probability for site percolation on $\mathbb{Z}^2$ (see \cite{Grimmett}, \cite{Kesten}). In case if both (\ref{8}) and (\ref{9}) are satisfied, what do we get?

After applying the $\theta (N)-$thresholding on the noisy picture $\{Y_{i,j}\}_{i,j=1}^{N}$, we obtained a (random) black-and-white image $\{\overline{Y}_{i,j}\}_{i,j=1}^{N}$. Let $\overline{G}_N$ be the graph of this image, as in Definition \ref{Definition2}.

Since $\overline{G}_N$ is a random, we actually observe the so-called \emph{site percolation} on black vertices within the subset of $\mathbb{Z}^2$. From this point, we can use results from percolation theory to predict formation of black and white clusters on $\overline{G}_N$, as well as to estimate the number of clusters and their sizes and shapes. Relations (\ref{8}) and (\ref{9}) are crucial here.

To explain this more formally, let us split the set of vertices $\overline{V}_N$ of the graph $\overline{G}_N$ into to groups: $\overline{V}_N = V_{N}^{im} \cup V_{N}^{out}$, where $V_{N}^{im} \cap V_{N}^{out} = \emptyset$, and $V_{N}^{im}$ consists of those and only those vertices that correspond to pixels belonging to the original object, while $V_{N}^{out}$ is left for the pixels from the background. Denote $G_{N}^{im}$ the subgraph of $\overline{G}_N$ with vertex set $V_{N}^{im}$, and denote $G_{N}^{out}$ the subgraph of $\overline{G}_N$ with vertex set $V_{N}^{out}$.

If (\ref{8}) and (\ref{9}) are satisfied, we will observe a so-called \emph{supercritical percolation} of black clusters on $G_{N}^{im}$, and a \emph{subcritical} percolation of black clusters on  $G_{N}^{out}$. Without going into much details on percolation theory (the necessary introduction can be found in \cite{Grimmett} or \cite{Kesten}), we mention that there will be a high probability of forming relatively large black clusters on $G_{N}^{im}$, but there will be only little and scarce black clusters on $G_{N}^{out}$. The difference between the two regions will be striking, and this is the main component in our image analysis method.

In this paper, mathematical percolation theory will be used to derive quantitative results on behaviour of clusters for both cases. We will apply those results to build  efficient randomized algorithms that will be able to detect and estimate the object $\{Im_{i,j}\}_{i,j=1}^{N}$ using the difference in percolation phases on $G_{N}^{im}$ and $G_{N}^{out}$.

If the \emph{noise level} $\sigma$ is not too large, then (\ref{8}) and (\ref{9}) are satisfied for some $\theta (N) \in (0,1)$. Indeed, one simply has to pick $\theta (N)$ close enough to 1. On the other hand, if $\sigma$ is relatively large, it may happen that (\ref{8}) and (\ref{9}) cannot both be satisfied at the same time.

\begin{definition}\label{Definition3}
In the framework defined by relations (\ref{1})-(\ref{3}) and assumptions $\langle A1 \rangle$ - $\langle A3 \rangle$, we say that the noise level $\sigma$ is \emph{small enough} (or \emph{1-small}), if the system of inequalities (\ref{8}) and (\ref{9}) is satisfied for some $\theta (N) \in \mathbb{R}$, for all $N \in \mathbb{N}$.
\end{definition}

A very important practical issue is that of choosing an optimal threshold value $\theta$. From a purely theoretical point of view, this is not a big issue: once (\ref{8}) and (\ref{9}) holds for some $\theta$, it is guaranteed that after $\theta -$thresholding we will observe \emph{qualitatively} different behaviour of black and white clusters in or outside of the true object. We will make use of this in what follows.

However, for practical computations, especially for moderate values of $N$, the value of $\theta$ is important. Since the goal is to make percolations on $V_{N}^{im}$ and $V_{N}^{out}$ look as different as possible, one has to make the corresponding percolation probabilities for black colour, namely,

\[
1 - F \biggr(\frac{\theta(N)}{\sigma}\biggr) \quad\mbox{and}\quad 1 - F \biggr(\frac{\theta(N) -1 }{\sigma}\biggr)\,,
\]

\noindent as different as possible both from each other and from the critical probability $p_{c}^{site}$. There can be several reasonable ways for choosing a suitable threshold. For example, we can propose to choose $\theta (N)$ as a maximizer of the following function:

\begin{equation}\label{10}
\biggr(1 - F \biggr(\frac{\theta(N)}{\sigma}\biggr) - p_{c}^{site}\biggr)^2 +\, \biggr(1 - F \biggr(\frac{\theta(N)-1}{\sigma}\biggr) - p_{c}^{site}\biggr)^2\,,
\end{equation}

\noindent provided that (\ref{8}) and (\ref{9}) holds. Alternatively, we can propose to use a maximizer of

\begin{equation}\label{11}
sign \biggr(1 - F \biggr(\frac{\theta(N)-1}{\sigma}\biggr) - p_{c}^{site}\biggr) \,+\, sign\biggr( p_{c}^{site} - 1 + F \biggr(\frac{\theta(N)}{\sigma}\biggr)\biggr)\,.
\end{equation}


\section{Object detection}\label{Section2}

%


We either observe a blank white screen with accidental noise or there is an actual object in the blurred picture. In this section, we propose an algorithm to make a decision on which of the two possibilities is true. This algorithm is a statistical testing procedure. It is designed to solve the question of testing $H_0:\, I_{ij}=0\,\, \mbox{for all}\,\, 1 \leq i, j \leq N$ versus $H_1:\, I_{ij} = 1 \,\,\mbox{for some}\,\, i, j$.

Let us choose $\alpha (N) \in (0, 1)$ - the \emph{probability of false detection} of an object. More formally, $\alpha (N)$ is the maximal probability that the algorithm finishes its work with the decision that there was an object in the picture, while in fact there was just noise. In statistical terminology, $\alpha (N)$ is the probability of an error of the first kind.

We allow $\alpha$ to depend on $N$; $\alpha(N)$ is connected with complexity (and expected working time) of our randomized algorithm.

Since in our method it is crucial to observe some kind of percolation in the picture (at least within the image), the image has to be "not too small" in order to be detectable by the algorithm: one can't observe anything percolation-alike on just a few pixels. We will use percolation theory to determine how "large" precisely the object has to be in order to be detectable. Some size assumption has to be present in any detection problem: for example, it is hopeless to detect a single point object on a very large screen even in the case of a moderate noise.

For an easy start, we make the following (way too strong) largeness assumptions about the image:\par\smallskip

$\langle \textbf{D1} \rangle \quad$ Assume that the image $\{\,(i, j)\, |\, 1 \leq i, j \leq N,\, I_{ij}=1\, \}$ contains a completely black square with the side of size at least $\varphi_{im}(N)$ pixels, where

\begin{equation}\label{12}
\lim_{N\to\infty} \frac{\,\log \frac{1}{\,\alpha(N)\,}\,}{\,\varphi_{im}(N)\,}\, = 0\,.
\end{equation}

\begin{equation}\label{14}
\langle \textbf{D2} \rangle \quad\quad \lim_{N \to\infty} \frac{\,\varphi_{im} (N)\,}{\,\log N\,} = \infty\,.\quad\quad\quad\quad\quad\quad\quad\quad\quad\quad\quad\quad\quad\quad\quad\quad\quad\quad\quad
\end{equation}


\noindent Furthermore, we assume the obvious consistency assumption

\begin{equation}\label{13}
\varphi_{im}(N)\, \leq N\,.
\end{equation}

\noindent Assumptions $\langle D1 \rangle$ and $\langle D2 \rangle$ are \emph{sufficient} conditions for our algorithm to work. They are way too strong for our purposes. It is possible to relax (\ref{14}) and to replace a square in $\langle D1 \rangle$ by a triangle-shaped figure.

Although the above two conditions are of asymptotic character, most of our estimates below are valid for finite $N$ as well. Nevertheless, it is important to remark here that asymptotic results for $N\to\infty$ also have interesting practical consequences. More specifically, assume that physically we always have screens of a fixed size, but the resolution $N^2$ of our cameras can grow unboundedly. When $N$ tends to infinity, we see that the same physical object that has, say, 1mm in width and in length, contains more and more pixels on the pixelized image. Therefore, for high-resolution pictures, our algorithm could detect fine structures (like nerves etc.) that are not directly visible by a human eye. \smallskip

Now we are ready to formulate our \emph{Detection Algorithm}. Fix the false detection rate $\alpha (N)$ before running the algorithm.\par\smallskip

\noindent \textbf{Algorithm 1 (Detection).}

\begin{itemize}
\item
Step 0. Find an optimal $\theta (N)$. \medskip

\item
Step 1. Perform $\theta(N)-$thresholding of the noisy picture $\{Y_{i,j}\}_{i,j=1}^{N}$. \medskip

\item
Step 2. Until\par\smallskip

\quad\quad\quad\quad\quad\quad\{\{Black cluster of size $\varphi_{im} (N) $ is found\}\par
\quad\quad\quad\quad or \par
\quad\quad\quad\quad\quad\quad\{all black clusters are found\}\}, \par\smallskip

\quad\quad\quad\quad Run depth-first search (\cite{Tarjan}) on the graph $\overline{G}_N$ of\par
\quad\quad\quad\quad\quad\quad\, the $\theta(N)-$thresholded picture $\{{\overline{Y}}_{i,j}\}_{i,j=1}^{N}$\medskip

\item
Step 3. If a black cluster of size $\varphi_{im} (N) $ was found, report that \par

\quad\quad\quad\quad an object was detected\bigskip

\item

Step 4. If no black cluster was larger than $\varphi_{im} (N) $, report that \par

\quad\quad\quad\quad there is no object.
\end{itemize}

\noindent At Step 2 our algorithm finds and stores not only sizes of black clusters, but also coordinates of pixels constituting each cluster. We remind that $\theta(N)$ is defined as in (\ref{6}), $\overline{G}_N$ and $\{{\overline{Y}}_{i,j}\}_{i,j=1}^{N}$ were defined in Section \ref{Section1}, and $\varphi_{im} (N)$ is any function satisfying (\ref{12}). The depth-first search algorithm is a standard procedure used for searching connected components on graphs. This procedure is a deterministic algorithm. The detailed description and rigorous complexity analysis can be found in \cite{Tarjan}, or in the classic book \cite{Aho_Hopcroft_Ullman}, Chapter 5.


Let us prove that Algorithm 1 works, and determine its complexity.

\begin{theorem}\label{Theorem1}
Let $\sigma$ be 1-small. Suppose assumptions $\langle D1 \rangle$ and $\langle D2 \rangle$ are satisfied. Then\smallskip

\begin{enumerate}
  \item Algorithm 1 finishes its work in $O(N^2)$ steps, i.e. is linear.\medskip

  \item If there was an object in the picture, Algorithm 1 detects it with probability at least $(1 - \exp(-C_1(\sigma) \varphi_{im} (N)))$.\medskip

  \item The probability of false detection doesn't exceed $\min\{\alpha (N), \exp(-C_2(\sigma) \varphi_{im} (N)) \}$ for all $N > N(\sigma)$.
\end{enumerate}

\noindent The constants $C_1 > 0$, $C_2 > 0$ and $N(\sigma)\in\mathbb{N}$ depend only on $\sigma$.

\end{theorem}

\begin{remark}\label{Remark10}
Dependence on $\sigma$ implicitly means dependence on $\theta(N)$ as well, but this doesn't spoil Theorem \ref{Theorem1}. Remember that we can consider $\theta(N)$ to be a function of $\sigma$ in view of our comments before (\ref{10}) and (\ref{11}).
\end{remark}

Theorem \ref{Theorem1} means that Algorithm 1 is of quickest possible order: it is \emph{linear} in the input size. It is difficult to think of an algorithm working quicker in this problem. Indeed, if the image is very small and located in an unknown place on the screen, or if there is no image at all, then any algorithm solving the detection problem will have to at least upload information about $O(N^2)$ pixels, i.e. under general assumptions of Theorem \ref{Theorem1}, any detection algorithm will have at least linear complexity.

Another important point is that Algorithm 1 is not only consistent, but that it has \emph{exponential} rate of accuracy.

\section{Proofs}\label{Section_Proofs}

This section is devoted to provide complete proofs of the above results. Some crucial estimates from percolation theory are also presented for the reader's convenience.

\begin{proof} (Theorem \ref{Theorem1}):

\noindent Part I. First we prove the complexity result. Finding a suitable (approximate, within a predefined error) $\theta$ from (\ref{10}) or (\ref{11}) takes a constant number of operations. See, for example, \cite{Krylov_Bobkov_Monastyrnyi}.

The $\theta(N)-$thresholding gives us $\{{\overline{Y}}_{i,j}\}_{i,j=1}^{N}$ and $\overline{G}_N$ in $O(N^2)$ operations. This finishes the analysis of Step 1.

As for Step 2, it is known (see, for example, \cite{Aho_Hopcroft_Ullman}, Chapter 5, or \cite{Tarjan}) that the standard depth-first search finishes its work also in $O(N^2)$ steps. It takes not more than $O(N^2)$ operations to save positions of all pixels in all clusters to the memory, since one has no more than $N^2$ positions and clusters. This completes analysis of Step 2 and shows that Algorithm 1 is linear in the size of input data.\par\smallskip

\noindent Part II. Now we prove the bound on the probability of false detection. Denote

\begin{equation}\label{15}
p_{out}(N)\,:=\,1 - F \biggr(\frac{\theta(N)}{\sigma}\biggr)\,,
\end{equation}

\noindent a probability of erroneously marking a white pixel outside of the image as black. Under assumptions of Theorem \ref{Theorem1}, $p_{out}(N) < p_{c}^{site}$.

We prove the following additional theorem:

\begin{theorem}\label{Theorem2}
Suppose that $0 < p_{out}(N) < p_{c}^{site}$. There exists a constant $C_3 = C_3 (p_{out}(N)) > 0$ such that

\begin{equation}\label{16}
P_{p_{out}(N)} (F_N(n))\,\leq\,\exp (\,-n \,C_3 (p_{out}(N)))\,,\quad\mbox{for all}\quad n \geq \varphi_{im} (N)\,.
\end{equation}

\noindent Here $F_N(n)$ is the event that there is an \emph{erroneously marked} black cluster of size greater or equal $n$, lying in the square of size $N \times N$ corresponding to the screen. (An erroneously marked black cluster is a black cluster on $G_N$ such that \emph{each} of the pixels in the cluster was wrongly coloured black after the $\theta-$thresholding.)
\end{theorem}

Before proving this result, we state the following theorem about subcritical site percolation.


\begin{theorem}(Aizenman-Newman)\label{Theorem3}
Consider site percolation with probability $p_0$ on $\mathbb{Z}^2$. There exists a constant $\lambda_{site} = \lambda_{site}(p_0) > 0$ such that

\begin{equation}\label{17}
P_{p_0} (\,|C| \geq n\,)\,\leq\, e^{-n\,\lambda_{site}(p_0)}\,,\quad\mbox{for all}\quad n \geq 1\,.
\end{equation}

\noindent Here $C$ is the open cluster containing the origin.
\end{theorem}

\begin{proof} (Theorem \ref{Theorem3}): See \cite{Bollobas_Riordan}.

\end{proof}

To conclude Theorem \ref{Theorem2} from Theorem \ref{Theorem3}, we will use the celebrated FKG inequality (see \cite{FKG}, or \cite{Grimmett}, Theorem 2.4, p.34; see also Grimmett's book for some explanation of the terminology).

\begin{theorem}\label{TheoremFKG}
If $A$ and $B$ are both increasing (or both decreasing) events on the same measurable pair $(\Omega , \mathcal{F})$, then $P(A \cap B)\,\geq\, P(A)\,P(B)\,$.
\end{theorem}
\medskip

\begin{proof} (Theorem \ref{Theorem2}):
Denote by $C(i,j)$ the largest cluster in the $N\times N$ screen containing the pixel with coordinates $(i, j)$, and by $C(0)$ the largest black cluster on the $N \times N$ screen containing 0. By Theorem \ref{Theorem3}, for all $i$, $j$: $1 \leq i, j \leq N$:

\begin{eqnarray}\label{19}
  P_{p_{out}(N)} (\,|C(0)| \geq n\,) & \leq & e^{-n\,\lambda_{site}(p_{out})}\,, \\
  P_{p_{out}(N)} (\,|C(i, j)| \geq n\,) & \leq & e^{-n\,\lambda_{site}(p_{out})}\,. \nonumber
\end{eqnarray}

\noindent Obviously, it only helped to inequalities (\ref{17}) and (\ref{19}) that we have limited our clusters to only a finite subset instead of the whole lattice $\mathbb{Z}^2$. On a side note, there is no symmetry anymore between arbitrary points of the $N \times N$ finite square; luckily, this doesn't affect the present proof.

Since $\{\,|C(0)| \geq n\,\}$ and $\{\,|C(i, j)| \geq n\,\}$ are increasing events (on the measurable pair corresponding to the standard random-graph model on $G_N$), we have that $\{\,|C(0)| < n\,\}$ and $\{\,|C(i, j)| < n\,\}$ are decreasing events for all $i$, $j$. By FKG inequality for decreasing events,

\begin{eqnarray*}
  P_{p_{out}(N)} ( \,|C(i, j)| < n\, \,\mbox{for all}\,\, i, j, 1 \leq i, j \leq N\,) & \geq &  \\
  \prod\prod_{1 \leq i, j \leq N} P_{p_{out}(N)} ( \,|C(i, j)| < n\,) & \geq & (\mbox{by (\ref{19})}) \\
   & \geq & {\bigr(\,1 - e^{-n\,\lambda_{site}(p_{out})}\,\bigr)^{N^2}}\,.
\end{eqnarray*}

\noindent We denote below by $C_{b}^{a}$ the "$a$ out of $b$" binomial coefficient. It follows that

\begin{eqnarray*}
  P_{p_{out}(N)} (F_N(n)) & = & P_{p_{out}(N)} \bigr( \,\exists (i, j),\, 1 \leq i, j \leq N\,: \,|C(i, j)| \geq n\, \bigr) \\
    & \leq & 1 -  {\bigr(\,1 - e^{-n\,\lambda_{site}(p_{out})}\,\bigr)^{N^2}} \\
    &=& 1 - \sum_{k=0}^{N^2} {(-1)}^k \,C_{N^2}^k \, e^{-n\,\lambda_{site}(p_{out})\,k} \\
    &=& \sum_{k=1}^{N^2} {(-1)}^{k-1} \,C_{N^2}^k \, e^{-n\,\lambda_{site}(p_{out})\,k} \\
    &=& N^2 e^{-n\,\lambda_{site}(p_{out})}\,+\,o\bigr(N^2 e^{-n\,\lambda_{site}(p_{out})}\bigr)\,,
\end{eqnarray*}

\noindent because we assumed in (\ref{16}) that $n \geq \varphi_{im} (N)$, and $\log N = o(\varphi_{im} (N))$. Moreover, we see immediately that Theorem \ref{Theorem2} follows now with some $C_3$ such that $0 < C_3 (p_{out}(N)) < \lambda_{site}(p_{out}(N))$.
\end{proof}

The exponential bound on the probability of false detection follows from Theorem \ref{Theorem2}.\par\smallskip

\noindent Part III. It remains to prove the lower bound on the probability of true detection. First we prove the following theorem:

\begin{theorem}\label{Theorem4}
Consider site percolation on $\mathbb{Z}^2$ lattice with percolation probability $p > p_c^{site}$. Let $A_n$ be the event that there is an open path in the rectangle $[0, n] \times [0, n]$ joining some vertex on its left side to some vertex on its right side. Let $M_n$ be the maximal number of vertex-disjoint open left-right crossings of the rectangle $[0, n] \times [0, n]$. Then there exist constants $C_4 = C_4 (p) > 0$, $C_5 = C_5 (p) > 0$, $C_6 = C_6 (p) > 0$ such that

\begin{equation}\label{20}
P_p (A_n) \,\geq\, 1 - n\,e^{-C_4\,n}\,,
\end{equation}

\begin{equation}\label{21}
P_p (\,M_n \leq C_5\,n\,) \,\leq\, e^{-C_6\,n}\,,
\end{equation}

\noindent and both inequalities holds for \emph{all} $n \geq 1$.
\end{theorem}

\begin{proof} (Theorem \ref{Theorem4}):
One proves this by a slight modification of the corresponding result for bond percolation on the square lattice. See proof of Lemma 11.22 and pp. 294-295 in \cite{Grimmett}.

\end{proof}

Now suppose that we have an object in the picture that satisfies assumptions of Theorem \ref{Theorem1}. Consider any $\varphi_{im} (N) \times \varphi_{im} (N)$ square in this image. After $\theta-$thresholding of the picture by Algorithm 1, we observe on the selected square site percolation with probability

\[
p_{im}(N) \,:=\, 1 - F \biggr(\frac{\theta(N)-1}{\sigma}\biggr) \,>\, p_{c}^{site}\,.
\]

\noindent Then, by (\ref{20}) of Theorem \ref{Theorem4}, there exists $C_4 = C_4 (p_{im}(N))$ such that there will be \emph{at least one} cluster of size not less than $\varphi_{im} (N)$ (for example, one could take any of the existing left-right crossings as a part of such cluster), provided that $N$ is bigger than certain $N(p_{im}(N)) = N(\sigma)$; and all that happens with probability at least

\[
1 - n\,e^{-C_4\,n} \,>\, 1 - e^{-C_3\,n}\,,
\]

\noindent for some $C_3$: $0 < C_3 < C_4$. Theorem \ref{Theorem1} is proved.

\end{proof}


\smallskip
\noindent {\bf Acknowledgments.} The authors would like to thank Laurie Davies, Remco van der Hofstad, Artem Sapozhnikov, Shota Gugushvili and Geurt Jongbloed for helpful discussions. \\

\bibliographystyle{plainnat}
\bibliography{Randomized_Algorithms_and_Percolation}


\end{document}